\documentclass[a4paper,10pt,reqno]{amsart}

\usepackage[utf8]{inputenc}
\usepackage[T1]{fontenc}
\usepackage{amssymb}
\usepackage[english]{babel}
\usepackage{mathrsfs}
\usepackage{tikz-cd}
\usepackage{verbatim}
\usepackage{color}
\usepackage[all]{xy}

\usepackage[shortlabels]{enumitem}
\setlist[enumerate]{label=\rm{(\arabic*)}}
\setlist[enumerate,2]{label=\rm{(\roman*)}}
\setlist[itemize]{label=\raisebox{0.25ex}{\tiny$\bullet$}}
\usepackage[backref, colorlinks, linktocpage, citecolor = blue, linkcolor = blue]{hyperref} 

\usepackage[left=2.5cm,right=2.5cm,top=2.5cm,bottom=2.5cm]{geometry}

\theoremstyle{plain}
\newtheorem{theorem}{Theorem}[section]
\newtheorem{corollary}[theorem]{Corollary}
\newtheorem{proposition}[theorem]{Proposition}
\newtheorem{lemma}[theorem]{Lemma}
\newtheorem{example}[theorem]{Example}

\newtheorem*{acknowledgments}{Acknowledgments}
\theoremstyle{definition}
\newtheorem{definition}[theorem]{Definition}
\newtheorem{remark}[theorem]{Remark}

\theoremstyle{plain}
\newtheorem{theoremA}{Theorem}

\newcommand{\rmap}{\dashrightarrow}
\newcommand{\psmap}{\smash{\xymatrix@C=0.5cm@M=1.5pt{ \ar@{..>}[r]& }}}
\def\NE{\overline{\operatorname{NE}}}
\def\Autzero{\mathrm{Aut}^\circ}
\def\Aut{\mathrm{Aut}}
\def\Bir{\mathrm{Bir}}

\def\PP{\mathbb{P}}
\def\FF{\mathbb{F}}

\def\PGL{\mathrm{PGL}}
\def\kk{\mathbf{k}}

\def\seg{\mathfrak{S}}

\begin{document}

\title[Connected algebraic subgroups of $\Bir(X)$ not contained in a maximal one]{Connected algebraic subgroups of groups of birational transformations not contained in a maximal one}
\author{Pascal Fong \and Sokratis Zikas}
\address{Universit\"at Basel, Departement Mathematik und Informatik, Spiegelgasse 1, CH--4051 Basel, Switzerland}
\email{pascal.fong@unibas.ch}
\address{Universit\"at Basel, Departement Mathematik und Informatik, Spiegelgasse 1, CH--4051 Basel, Switzerland}
\email{sokratis.zikas@unibas.ch}

\begin{abstract}
	We prove that for each $n\geq 2$, there exist a ruled variety $X$ of dimension $n$ and a connected algebraic subgroup of $\Bir(X)$ which is not contained in a maximal one.
\end{abstract}

\maketitle

\section{Introduction}

Let $\kk$ be an algebraically closed field. The classification of algebraic subgroups of groups of birational transformations was initiated in \cite{Enriques}, where Enriques shows that each connected algebraic subgroup of $\Bir(\PP^2)$ is conjugate to an algebraic subgroup of $\Autzero(S)$, with $S$ isomorphic to $\PP^2$ or to the $n$-th Hirzebruch surface $\FF_n$ for $n\neq 1$; and these are all maximal, with respect to the inclusion, among the connected algebraic subgroups of $\Bir(\PP^2)$. The connected algebraic subgroups of $\Bir(\PP^3)$ have been classified over $\kk=\mathbb{C}$ by Umemura in a series of four papers \cite{Umemura80,Umemura82a,Umemura82b,Umemura85} and it follows again from his classification that each connected algebraic subgroup of $\Bir(\PP^3)$ is contained in a maximal one (see also \cite{BFT1, BFT2} for a modern approach). However, it is an open problem whether every connected algebraic subgroup of $\Bir(\PP^n)$ is contained in a maximal one when $n\geq 4$. 

On the other hand, it is proven in \cite[Theorem C]{fong} that there exist connected algebraic subgroups of $\Bir(C\times \PP^1)$ not contained in a maximal one when $C$ is a smooth curve of positive genus. The proof of this result is based on the existence of infinite increasing sequences of connected algebraic subgroups of $\Bir(C\times \PP^1)$ (see \cite[Theorem A]{fong}), and on the fact that the dimension of a maximal connected algebraic subgroup of $\Bir(C\times \PP^1)$ is bounded by $4$ (see \cite[Theorem B]{fong} and \cite[Theorem 3]{Maruyama}). Our main result in this note is a higher dimensional analogue of \cite[Theorem C]{fong}:

\begin{theoremA}\label{Theorem}
	Let $\kk$ be an algebraically closed field of characteristic $0$. Let $n\geq 1$ and $C$ be a smooth curve of positive genus. 
	Then there exists a connected algebraic subgroup of  $\Bir(C\times \PP^n)$  which is not contained in a maximal one.
\end{theoremA}

The idea of the proof is to consider the connected algebraic subgroup $\Autzero(S\times \PP^n)$, where $S$ is a ruled surface such that $\Autzero(S)$ is not contained in a maximal connected algebraic subgroup of $\Bir(S)$, and to show that it cannot be contained in a maximal connected algebraic subgroup of $\Bir(S\times \PP^{n})$. Since $\Autzero(S\times \PP^n)\simeq \Autzero(S)\times \PGL_{n+1}(\kk)$ by \cite[Corollary 4.2.7]{BSU}, the existence of infinite increasing sequences of connected algebraic subgroups of $\Bir(C\times \PP^{n+1})$ is an immediate consequence of \cite[Theorem A]{fong}. From this alone, it is nonetheless insufficient to deduce that one of the connected algebraic subgroups of $\Bir(C\times \PP^{n+1})$ appearing in the infinite increasing sequences is not contained in a maximal one (see Remark \ref{infiniteseq}), and classifying all connected algebraic subgroups of $\Bir(C\times \PP^{n+1})$ seems out of reach at the moment.

This article is organized as follows. Section $2$ contains two results, namely Lemmas \ref{lemmadim2} and \ref{lemma2dim2}, which are important for the proof of the higher dimensional case. As a consequence of these two lemmas, we also get a new and short proof of the dimension two case (see Proposition \ref{dim2}), without using the classification of the maximal connected algebraic subgroups of $\Bir(C\times \PP^1)$ (\cite[Theorem B]{fong}). In Section $3$, we prove the higher dimensional case under the extra assumption that $\mathrm{char}(\kk)=0$, in view of using the machinery of the MMP and the $G$-Sarkisov program. The latter has been developped by Floris in \cite{Floris}, building upon results of Hacon and McKernan in \cite{HM13}. More precisely, if $G$ is a connected algebraic group, then every $G$-equivariant birational map between Mori fibre spaces decomposes into $G$-Sarkisov links (see \cite[Theorem 1.2]{Floris}). We study the possible links in Lemmas \ref{SxPP^n} and \ref{S'xPP^n}. Combining Proposition \ref{dim2} and Theorem \ref{higherdim}, we get Theorem \ref{Theorem}. 

It is very natural to also ask whether for all $n\geq 2$, there exists a variety $X$ of dimension $n$ such that $\Bir(X)$ contains algebraic subgroups which are not lying in a maximal one, without the connectedness assumption. If $n=2$, the answer is also affirmative (see \cite[Lemma 3.1, Corollary B]{fong2}), and the proof is analogous to that of the connected case. Since the $G$-Sarkisov program is known only for connected algebraic groups, it is not clear if the proof presented in this article could be adapted for the non-connected case in higher dimension.

\begin{acknowledgments}
	We are thankful to Jérémy Blanc, Michel Brion, Enrica Floris, Ronan Terpereau, Susanna Zimmermann for interesting discussions and remarks. Special thanks to Enrica Floris for pointing out a mistake in the proof of a preliminary version. We are also grateful to the anonymous referees for their careful reading and useful comments. The first and second authors respectively acknowledge support by the Swiss National Science Foundation Grants “Geometrically ruled surfaces” 200020–192217 and “Birational transformations of threefolds” 200020-178807.
\end{acknowledgments}

\section{Some preliminaries and the case of dimension two}

From now on, $C$ will always denote a smooth curve of genus $g$ over a field $\kk$. In this section, $\kk$ is an algebraically closed field of arbitrary characteristic. The following invariant was used by Maruyama in \cite{Maruyamabook, Maruyama} for his classification of ruled surfaces and their automorphisms.

\begin{definition}
	Let $V$ be a rank-$2$ vector bundle over $C$ and $\tau\colon S=\PP(V)\to C$ be a ruled surface. We say that $\tau$ is \emph{decomposable} if $V$ is the direct sum of two line bundles over $C$. Otherwise, we say that $\tau$ is \emph{indecomposable}. We define the \emph{Segre invariant} of $S$ as $$\seg(S) = \min\{\sigma^2, \sigma\text{ section of } \tau\}.$$
\end{definition}

\begin{remark}\label{remarkruled}
	Let $\tau \colon S\to C$ be a ruled surface.
	\begin{enumerate}
		\item Let $p\in S$ and $\sigma$ be a section of $\tau$. Recall that the blow-up of $S$ at $p$ followed by the contraction of the strict transform of the fibre passing through $p$, yields a ruled surface $\tau'\colon S'\to C$ and a birational map $\epsilon\colon S\dashrightarrow S'$ called the \emph{elementary transformation of $S$ centered at p} (see e.g.\ \cite[V. Example 5.7.1]{Hartshorne}). Let $\sigma'$ be the strict transform of $\sigma$ by $\epsilon$. If $p\in \sigma$, then $\sigma'^2 = \sigma^2 -1$. Else, $\sigma'^2 = \sigma^2 +1$.
		\item As $S$ is obtained by finitely many elementary transformations from $C\times \PP^1$ (see e.g.\ \cite[V. Exercise 5.5]{Hartshorne}) and $\seg(C\times\PP^1)=0$ (see e.g.\ \cite[Lemma 2.14]{fong}), it follows that $\seg(S)>-\infty$. If moreover $\seg(S)<0$, then there exists a unique section with negative self-intersection number (see e.g.\ \cite[Lemma 2.10. (1)]{fong2}).
		\item \label{remarkruled.3}The Segre invariant $\seg(S)$ equals $-e$, where $e$ is the invariant defined in \cite[V. Proposition 2.8]{Hartshorne}. If $\tau$ is indecomposable, then by \cite[V. Theorem 2.12. (b)]{Hartshorne}, we get $\seg(S)\geq 2-2g=-\deg(K_C)$. In particular, if $\seg(S)< -\deg(K_C)$, then $\tau$ is decomposable. 
	\end{enumerate}
\end{remark}

We recall the statement of Blanchard's lemma and its corollary (see \cite[Proposition 4.2.1, Corollary 4.2.6]{BSU}):

\begin{proposition}\label{blanchard}
	Let $f\colon X\to Y$ be a proper morphism of schemes such that $f_*(\mathcal{O}_X) = \mathcal{O}_Y$, and let $G$ be a connected group scheme acting on $X$. Then there exists a unique action of $G$ on $Y$ such that $f$ is $G$-equivariant.
\end{proposition}

\begin{corollary}
	Let $f\colon X\to Y$ be a proper morphism of projective schemes such that $f_*(\mathcal{O}_X) = \mathcal{O}_Y$. Then $f$ induces a homomorphism of group schemes $f_*\colon \Autzero(X)\to \Autzero(Y)$.
\end{corollary}

\begin{remark}\label{blanchardtrivial}
	Let $\tau\colon S\to C$ be a decomposable ruled surface. Assume that $C$ has genus $g=1$ and $\seg(S)\neq 0$, or that $g\geq 2$. Then by \cite[Lemma 7]{Maruyama}, the morphism induced by Blanchard's lemma $\tau_*\colon \Autzero(S)\to \Autzero(C)$ is trivial.
\end{remark}

In the next two lemmas, we compute $\Autzero(S)$ and its orbits for a ruled surface $\tau\colon S\to C$ with $\seg(S)<-(1+\deg(K_C))$ (which is decomposable by Remark \ref{remarkruled} \ref{remarkruled.3}). 

\begin{lemma}\label{lemmadim2}
	Let $C$ be a curve of genus $g\geq 1$. Let $\tau\colon S=\PP(V)\to C$ be a decomposable $\PP^1$-bundle such that $\seg(S)<-(1+\deg(K_C))$. Let $\sigma$ be the minimal section of $\tau$ and $L(\sigma)$ be the line subbundle of $V$ associated to $\sigma$. We choose trivializations of $\tau$ such that $\sigma$ is the infinity section. Then the following hold:
	\begin{enumerate}
		\item \label{lemmadim2.1} The group $\Autzero(S)$ is isomorphic to $\mathbb{G}_m \rtimes \Gamma(C,\det(V)^{\vee} \otimes L(\sigma)^{\otimes 2})$, where $\det(V)$ denotes the determinant line bundle of $V$. This isomorphism associates $\alpha\in \mathbb{G}_m$ and $\gamma\in \Gamma(C,\det(V)^{\vee} \otimes L(\sigma)^{\otimes 2})$, to the element $\mu_{\alpha,\gamma}\in \Autzero(S)$ obtained by gluing the automorphisms:
		\begin{align*}
			U_i \times \PP^1 & \to U_i\times \PP^1 \\
			(x,[u:v]) & \mapsto 
			\left(x, [\alpha u + \gamma_{|U_i}(x) v :v]
			\right).
		\end{align*}
		\item \label{lemmadim2.2} The $\Autzero(S)$-orbits in $S$ are $\{p\}$ and $\tau^{-1}(\tau(p)) \setminus \{p\}$ for $p\in \sigma$.
	\end{enumerate}
\end{lemma}

\begin{proof}
	\begin{enumerate}[wide]
		\item The proof follows from the computation made in \cite[case (b) p.92]{Maruyama}. For the sake of self-containess, we recall it below. Since $\tau$ is decomposable, we can write its transition maps as $t_{ij}\colon U_j\times \PP^1 \to U_i\times \PP^1$, $(x,[u:v])\mapsto \left(x,[a_{ij}(x)u:b_{ij}(x)v]\right)$, where $[u:v]$ denotes the coordinates of $\PP^1$, $a_{ij}\in \mathcal{O}_C(U_i\cap U_j)^*$ denotes the transition maps of the line bundle $L(\sigma)$ and $b_{ij}\in \mathcal{O}_C(U_i\cap U_j)^*$. Let $\mu\in \Autzero(S)$. The morphism induced by Blanchard's lemma $\tau_*\colon \Autzero(S)\to \Autzero(C)$ is trivial (Remark \ref{blanchardtrivial}). Moreover, $\sigma$ is fixed by $\Autzero(S)$ as it is the unique minimal section. Therefore, for each trivializing open subset $U_i\subset C$, $\mu$ induces an automorphism $\mu_i\colon U_i\times \PP^1 \to U_i\times \PP^1$, given by $(x,[u:v])\mapsto \left(x, [\alpha_i(x) u + \gamma_i(x) v:v] \right)$, where $\alpha_i\in \mathcal{O}_C(U_i)^*$ and $\gamma_i\in \mathcal{O}_C(U_i)$. The condition $\mu_i t_{ij} = t_{ij}\mu_j$ implies that $\alpha_i = \alpha_j = \alpha\in \mathbb{G}_m$ and $\gamma_i = b_{ij}^{-1}a_{ij}\gamma_j$. Since $a_{ij}b_{ij}$ are the transition maps of the line bundle $\det(V)$, and $a_{ij}$ denote the transition maps of $L(\sigma)$, it implies that $\gamma\in \Gamma(C,\det(V)^{\vee} \otimes L(\sigma)^{\otimes 2})$. The data of $\alpha \in \mathbb{G}_m$ and $\gamma\in \Gamma(C,\det(V)^{\vee} \otimes L(\sigma)^{\otimes 2})$ determine uniquely the automorphism $\mu$, this proves that we have an embedding $\Autzero(S)\hookrightarrow \mathbb{G}_m \rtimes \Gamma(C,\det(V)^{\vee} \otimes L(\sigma)^{\otimes 2})$. Conversely, one can check that the automorphisms defined in the statement commute with the transition maps, hence their gluing defines an automorphism of $S$. Because $\mathbb{G}_m \rtimes \Gamma(C,\det(V)^{\vee} \otimes L(\sigma)^{\otimes 2})$ is also connected, we get that it is isomorphic to $\Autzero(S)$.
		\item Since the morphism induced by Blanchard's lemma $\tau_*\colon \Autzero(S)\to \Autzero(C)$ is trivial (Remark \ref{blanchardtrivial}), each $\Autzero(S)$-orbit is contained in a fibre of $\tau$. As $\sigma$ is the unique section with negative self-intersection number, it is fixed pointwise by $\Autzero(S)$. It remains to see that $\Autzero(S)$ acts transitively on $\tau^{-1}(\tau(p))\setminus \{p\}$ for each $p$ lying on $\sigma$. 
		
		Let $L=\det(V)^{\vee} \otimes L(\sigma)^{\otimes 2}$. It follows from  \cite[Proposition 2.15]{fong} that $\deg(L)=-\seg(S)>1+\deg(K_C)$. 
		Let $p\in \sigma$ and let $\tau(p)=z$. We get by Serre duality that 
		\[
		h^1(C,L) = h^0(C,K_C \otimes L^{\vee}) = 0,
		\]
		where the last equality follows from the fact that $\deg(K_C \otimes L^{\vee})	<	-1$.
		Similarly we get the equality 	$h^1(C,L\otimes \mathcal{O}_C(z)^{\vee})=0$. 
		By Riemann-Roch, $h^0(C,L\otimes \mathcal{O}_C(z)^{\vee}) =\deg(L) -g < \deg(L)-g+1 =h^0(C,L)$. Therefore, $z$ is not a base point of the complete linear system $\vert L \vert$, i.e.\ there exists $\gamma \in H^0(C,L)$ such that $\gamma(z)\neq 0$, and the subgroup $\mathbb{G}_a\simeq \{\mu_{1,\lambda\gamma};\lambda\in \kk\} $ acts transitively on $\tau^{-1}(z)\setminus \{p\}$ (see \ref{lemmadim2.1} for the definition of $\mu_{1,\lambda \gamma}$).
	\end{enumerate}
\end{proof}

Let $S$ be a ruled surface as in Lemma \ref{lemmadim2}, and $\phi\colon S\dashrightarrow S'$ be an $\Autzero(S)$-equivariant birational map. 
In the following lemma, we compute the fixed points of the action of $\phi \Autzero(S)\phi^{-1}$ on $S'$.

\begin{lemma}\label{lemma2dim2}
	Let $C$ be a curve of genus $g\geq 1$. Let $\tau\colon S\to C$ be a decomposable $\PP^1$-bundle such that $\seg(S)<-(1+\deg(K_C))$. If $\tau'\colon S'\to C$ is a ruled surface and there exists an $\Autzero(S)$-equivariant birational map $\phi\colon S\dashrightarrow S'$ which is not an isomorphism, then $\seg(S')<\seg(S)$ and $\phi \Autzero(S) \phi^{-1} \subsetneq \Autzero(S')$. The fixed points of the action of $\phi\Autzero(S)\phi^{-1}$ on $S'$ are the points lying on the minimal section of $\tau'$ and the base points of $\phi^{-1}$. Moreover, we can write $\phi$ as a product of $\Autzero(S)$-equivariant elementary transformations centered on the minimal sections.
\end{lemma} 

\begin{proof}
	By \cite[Theorem 7.7]{DI}, we can write $\phi = \phi_n \cdots \phi_1$ where each $\phi_i$ is an $\Autzero(S)$-equivariant elementary transformation. Without loss of generality, we can assume that this decomposition is minimal (i.e.\ the number of elementary transformations $n$ is minimal among all possible factorizations), and we prove the statement by induction on $n\geq 1$.
	
	Let $\sigma$ be the minimal section of $\tau$. By Lemma \ref{lemmadim2} \ref{lemmadim2.2}, the algebraic group $\Autzero(S)$ acts transitively on $\tau^{-1}(\tau(p))\setminus \{p\}$ for every $p\in \sigma$. Since $\phi_1$ is $\Autzero(S)$-equivariant, it follows that $\phi_1\colon S\dashrightarrow S_1$ is an elementary transformation centered on a point $p_1\in \sigma$. The strict transform of $\sigma$ by $\phi_1$ is the minimal section $\sigma_1$ of the ruled surface $\tau_1\colon S_1\to C$, and so $\seg(S_1)=\seg(S)-1$. Since the base point $q_1$ of $\phi_1^{-1}$ does not lie on the minimal section $\sigma_1$ of $\tau_1$, it follows by Lemma \ref{lemmadim2} \ref{lemmadim2.2} that $q_1$ is not fixed by $\Autzero(S_1)$. 
	Since $q_1$ is fixed by $\phi_1 \Autzero(S)\phi_1^{-1}$, we have the strict inclusion $\phi_1 \Autzero(S)\phi_1^{-1}\subsetneq \Autzero(S_1)$.
	In the complement of the fibres $f_{p_1}\subset S$ and $f_{q_1}\subset S_1$ containing the points $p_1$ and $q_1$ respectively, $\phi_1$ is an isomorphism. 	
	Therefore, by Lemma \ref{lemmadim2}, the only fixed points of $\phi_1 \Autzero(S)\phi_1^{-1}$ that lie in the complement of $f_{q_1}$ are the points on the minimal section $\sigma_1$.	
	It remains to check that the only fixed points on $f_{q_1}$ are the point $q_1'\in \sigma_1$ and the base point $q_1$ of $\phi^{-1}$.
	Let $U$ be a trivializing open subset of $\tau$ with $\tau(p_1)\in U$, and let $f\in \mathcal{O}_C(U)$ such that $\mathrm{div}(f)_{|U}=\tau(p_1)$. 
	We also choose trivializations of $\tau$ such that $\sigma$ is the infinity section.
	Up to isomorphisms at the source and the target, ${\phi_1}_{|U}$ equals $(x,[u:v])\mapsto (x,[f(x)u:v])$. By Lemma \ref{lemmadim2} \ref{lemmadim2.1}, there is an action of $\mathbb{G}_m$ on $S$ given locally by $(x,[u:v])\mapsto (x,[\alpha u:v])$. It implies that there is an action of $\phi_1 \mathbb{G}_m \phi_1^{-1}$ on $S_1$, given locally by $(x,[u:v])\mapsto (x,[\alpha f(x)u:f(x)v])=(x,[\alpha u :v])$. Therefore, $\phi_1 \mathbb{G}_m \phi_1^{-1}\subset \Autzero(S')$ acts transitively on $ f_{q_1}\setminus \{q_1,q_1'\}$. Since $\phi_1 \Autzero(S)\phi_1^{-1}\subset \Autzero(S')$ acts fibrewise (Remark \ref{blanchardtrivial}) and is connected, we get that $q_1$ and $q_1'$ are the fixed points of the action of $\phi_1 \Autzero(S) \phi_1^{-1}$ on $f_{q_1}$.
	
	Assume the statement holds for the birational map $\psi =\phi_{i} \cdots \phi_1\colon S\dashrightarrow S_i$, for some $i\geq 1$, and where $\tau_i\colon S_i \to C$ is a ruled surface with a minimal section $\sigma_i$. We now prove that the statement is then true for $\phi_{i+1}\psi$. By induction, the fixed points of $\psi \Autzero(S)\psi^{-1}$ on $S_i$ are the points lying on the minimal section $\sigma_i$ and the base points of $\psi^{-1}$.
	 Assume that $\phi_{i+1}$ is centered on a base point of $\psi^{-1}$, which is (the image of) the base point of the inverse of a previous elementary transformation $\phi_j$. A local calculation yields that we may cancel both $\phi_j$ and $\phi_{i+1}$, which contradicts the minimality of the factorization of $\phi$. So $\phi_{i+1}$ is centered on a point lying on the minimal section $\sigma_i$. Hence $\seg(S_{i+1}) = \seg(S_i)-1<\seg(S)$ by induction, and $\phi_{i+1}(\psi \Autzero(S) \psi^{-1})\phi_{i+1}^{-1} \subset \Autzero(S_{i+1}) $. The base point of $\phi_{i+1}$ is fixed by $\phi_{i+1}(\psi \Autzero(S) \psi^{-1})\phi_{i+1}^{-1}$, but is not fixed by $\Autzero(S_i)$ (by Lemma \ref{lemmadim2}). Thus, we get the strict inclusion $\phi_{i+1}(\psi \Autzero(S) \psi^{-1})\phi_{i+1}^{-1} \subsetneq \Autzero(S_{i+1}) $.
\end{proof}

The infinite increasing sequences of automorphism groups given in \cite[Theorem A]{fong} can be obtained from Lemma \ref{lemma2dim2}, but they do not imply that $\Autzero(S)$ is not contained in a maximal connected algebraic subgroup. As it is explained below, we can get an infinite increasing sequence of connected algebraic subgroups, where each of them is included in a maximal one, which a fortiori cannot be the same for all of them.

\begin{remark}\label{infiniteseq}
	Let $n\geq d\geq 2$. Define the connected algebraic groups $$G_d=\{\mathbb{A}^2\to \mathbb{A}^2, (x,y)\mapsto (x,y+p(x)),p\in \kk[x]_{\leq d}\},$$
	acting regularly on $\mathbb{A}^2$, and then birationally on $\PP^2$ via any embedding $\mathbb{A}^2 \hookrightarrow \PP^2$. Then $G_d\subsetneq G_{d+1}$ for all $d$. On the other hand, using an explicit description of $\Autzero(\FF_n)$ from \cite[\S 4.2]{Blanc}, we get for all $n\geq d$ that $G_d$ is a subgroup of $\Autzero(\FF_n)$, which is a maximal connected algebraic subgroup of $\Bir(\PP^2)$.
\end{remark}

Notice that for any variety $X$, using Remark \ref{infiniteseq}, we may produce an infinite increasing sequence of connected algebraic subgroups of $\Bir(X \times \PP^2)$. In particular, for $n\geq 2$ and $C$ a curve of positive genus, the same is true for $\Bir(C\times \PP^n)\simeq \Bir(C\times \PP^{n-2} \times \PP^2)$.

We reprove below partially \cite[Theorem C]{fong}, without using \cite[Theorem B]{fong}.

\begin{proposition}\label{dim2}
	Let $C$ be a curve of genus $g\geq 1$ and let $\tau\colon S\to C$ be a decomposable $\PP^1$-bundle such that $\seg(S)<-(1+\deg(K_C))$. Then $\Autzero(S)$ is not contained in a maximal connected algebraic subgroup of $\Bir(S)$.
\end{proposition}

\begin{proof}
	Assume that $\Autzero(S)$ is contained in a maximal connected algebraic subgroup $G$ of $\Bir(S)$. Then $G$ acts regularly on a surface $Y$ by Weil regularization theorem (see \cite{Weil}, or \cite{Zaitsev,Kraft} for a modern proof). By \cite[Corollary 3]{Brionnormal}, we can choose $Y$ to be normal and projective. Using an equivariant resolution of singularities (see \cite[Remark B, p.155]{Lipman}), we can also assume $Y$ to be smooth. Then by Blanchard's lemma (see Proposition \ref{blanchard}), the successive contractions of the $(-1)$-curves gives rise to a ruled surface $S'$ such that the induced birational morphism $Y\to S'$ is $G$-equivariant. Since $G$ is maximal and connected, it follows that $G\simeq \Autzero(S')$. The induced birational map $\phi\colon S\dashrightarrow S'$ is $\Autzero(S)$-equivariant. If $\phi$ is an isomorphism, then $\seg(S)= \seg(S')$. Else $\phi$ factorises as product of $\Autzero(S)$-equivariant elementary transformations centered on the minimal sections and $\seg(S')<\seg(S)$ (by Lemma \ref{lemma2dim2}). In both cases, we have $\seg(S')\leq \seg(S)$. Let $\epsilon\colon S'\dashrightarrow S''$ be an elementary transformation centered on the minimal section of $\tau'\colon S'\to C$. Then again by Lemma \ref{lemma2dim2}, it follows that $\epsilon \Autzero(S') \epsilon^{-1}\subsetneq \Autzero(S'')$, which contradicts the maximality of $G$ as a connected algebraic subgroup of $\Bir(S)$.
\end{proof}

\section{Higher dimensional case}

In what follows, we would like to utilize the machinery of the $G$-Sarkisov program for a connected algebraic group $G$.
Thus from now on, we furthermore assume that $\mathrm{char}(\kk) = 0$. The $G$-Sarkisov program is a non-deterministic algorithm that decomposes every $G$-equivariant birational map between two $G$-Mori fibre spaces as a product of simpler maps called $G$-Sarkisov links.
Its non-equivariant version was proven by Hacon and McKernan in \cite{HM13} and, building on their result, Floris proved the $G$-equivariant version in \cite{Floris}.
We follow the strategy of the proof of Proposition \ref{dim2}, and in view of using $G$-Sarkisov program, we recall first the definition:

\begin{definition}\label{Sarkisov link}
	Let $G$ be a connected algebraic group. A \emph{$G$-Mori fibre space} is a Mori fibre space with a regular action of $G$. Let $\pi_1 \colon X_1 \to B_1$ and $\pi_2 \colon X_2 \to B_2$ be two birational $G$-Mori fibre spaces.
	A \emph{$G$-Sarkisov diagram} between $X_1/B_1$ and $X_2/B_2$ is a commutative diagram of the form
	\[
	\xymatrix{
		Y_1 \ar[d]_{\alpha_1} \ar@{..>}[rr]^{\chi}  && Y_2 \ar[d]^{\alpha_2}\\
		X_1 \ar[d]_{\pi_1} && X_2 \ar[d]^{\pi_2} \\
		B_1 \ar[dr]_{s_1} & & B_2\ar[dl]^{s_2}\\
		& R
	}
	\]
	which satisfies the following properties:
	\begin{enumerate}
		\item  \label{prop:MMP} all morphisms appearing in the diagram are either isomorphisms or outputs of some $G$-equivariant MMP on a $\mathbb{Q}$-factorial klt $G$-pair $(Z,\Phi)$ (recall that a $G$-pair is a pair $(Z,\Phi)$ such that $G$ acts regularly on $Z$ and there is an induced regular action on $\Phi$),
		\item \label{prop:singularities} maximal dimensional varieties have $\mathbb{Q}$-factorial and terminal singularities,
		\item $\alpha_1$ and $\alpha_2$ are $G$-equivariant divisorial contractions or isomorphisms,
		\item $s_1$ and $s_2$ are $G$-equivariant extremal contractions or isomorphisms,
		\item $\chi$ is an isomorphism or a composition of $G$-equivariant anti-flips/flop/flips (in that order),
		\item the relative Picard rank $\rho(Z/R)$ of any variety $Z$ in the diagram is at most $2$. \label{prop:5}		
	\end{enumerate}
	We call $R$ the \emph{base} of the diagram. 
	
	Property \ref{prop:5} implies that $\alpha_1$ is a divisorial contraction if and only if $s_1$ is an isomorphism. A similar statement holds for the right hand side of the diagram. Depending whether ${s_1}$ or $s_2$ is an isomorphism, we get four types of Sarkisov diagrams: \vspace{0.3cm}
	
	\noindent
	\begin{minipage}{.25\linewidth}
		\begin{center}
			\emph{Type I}
		\end{center}
		\[
		\xymatrix{
			Y_1 \ar@{..>}[r] \ar[d] & X_2 \ar[d]^{\pi_2} \\
			X_1  \ar[d]_{\pi_1} & B_2 \ar[dl]\\
			*+[r]{B_1=R}
		}
		\]
	\end{minipage}\hspace{-0.5cm}
	\begin{minipage}{.25\linewidth}
		\begin{center}
			\emph{Type II}
		\end{center}
		\[
		\xymatrix@C=0.2cm{
			Y_1 \ar@{..>}[rr] \ar[d] && Y_2 \ar[d]\\
			X_1  \ar[d]_{\pi_1} && X_2 \ar[d]^{\pi_2} \\
			*+[r]{B_1=}& R &*+[l]{=B_2}
		}
		\]
	\end{minipage}\hspace{-0.5cm}
	\begin{minipage}{.25\linewidth}
		\begin{center}
			\emph{Type III}
		\end{center}
		\[
		\xymatrix{
			X_1 \ar@{..>}[r] \ar[d]_{\pi_1} & Y_2 \ar[d] \\
			B_1  \ar[rd] & X_2 \ar[d]^{\pi_2}\\
			&*+[l]{R=B_2}
		}
		\]
	\end{minipage}\hspace{-0.2cm}
	\begin{minipage}{.25\linewidth}
		\begin{center}
			\emph{Type IV}
		\end{center}
		\[
		\xymatrix@C=.4cm{
			X_1 \ar@{..>}[rr] \ar[d]_{\pi_1} && X_2 \ar[d]^{\pi_2} \\
			B_1  \ar[rd] && B_2 \ar[ld]\\
			&R &.
		}
		\]
	\end{minipage} \vspace{0.3cm}\\
	The birational map $\psi = \alpha_2 \chi \alpha_1^{-1}$ between $X_1$ and $X_2$ is called a \emph{$G$-Sarkisov link}. 
\end{definition}

\begin{remark}
	Property \ref{prop:singularities} does not follow directly from the original definition of a ($G$-)Sarkisov diagram of \cite{HM13} and \cite{Floris}.
	For a proof, see \cite[Proposition 4.25]{BLZ}.
\end{remark}

In subsequent proofs we are going to make heavy use of the following elementary but useful observation:

\begin{remark}\label{rank 2 fibrations}
	Let $Z$ be one of the varieties appearing in a $G$-Sarkisov diagram, such that the relative Picard rank $\rho(Z/R)$ is $2$. Then the $G$-Sarkisov diagram is uniquely determined by the datum of $Z \to R$, by a process known as the \emph{$2$-ray game} (see \cite[section 2.F]{BLZ}).
	
	More specifically, the $2$-ray game is a deterministic process that assigns to any such $Z \to R$ a $G$-Sarkisov diagram.
	Moreover any $G$-Sakrisov diagram can be recovered by the $2$-ray game on any of its relative Picard rank $2$ morphisms.
	Thus, up to orientation of the diagram, there is a unique $G$-Sarkisov diagram that contains $Z \to R$.
\end{remark}

\begin{lemma}\label{SxPP^n}
	Let $n\geq 1$ and $C$ be a curve of genus $g\geq 1$. 
	Let $\tau\colon S\to C$ be a decomposable $\PP^1$-bundle such that $\seg(S)< -(1+\deg(K_C))$ with minimal section $\sigma$ and let $\phi\colon S \rmap S'$ be an $\Autzero(S)$-equivariant birational map (possibly the identity) to a $\PP^1$-bundle $\tau'\colon S' \to C$. 
	Let $\pi'=\tau'\times id_{\PP^n}\colon S'\times \PP^n \to C\times \PP^n$ and $\pi_1'\colon S'\times \PP^n \to S'$ be the projection to the first factor. 
	Then the following hold:
	\begin{enumerate}
		\item \label{3.5.1} The only non-trivial $\Autzero(S\times \PP^n)$-Sarkisov diagrams, where $\pi'\colon S'\times \PP^n\to C\times \PP^n$ is the LHS Mori fibre space, are the following ones:
		\[
		\begin{tikzcd} [column sep=3em,row sep = 3em]
			T\times \PP^n\arrow[d,"\alpha",swap]\arrow[r,equal] & T\times \PP^n\arrow[d,"\beta"] &   S'\times \PP^n\arrow[d,"\pi'",swap]\arrow[rr,equal] && S'\times \PP^n\arrow[d,"\pi_1'"]  \\
			S'\times \PP^n\arrow[d,"\pi'",swap] & S''\times \PP^n\arrow[d,"\pi''"]& C\times \PP^n\arrow[rd,"p_1",swap] && S'\arrow[ld,"\tau'"] \\
			C\times \PP^n\arrow[r,equal] & C\times \PP^n && C&.
		\end{tikzcd}
		\] 
		In the first case, the induced Sarkisov link $S'\times \PP^n \rmap S'' \times \PP^n$ is equal to $\psi \times id_{\PP^n}$, where $\psi\colon S' \dashrightarrow S''$ is an elementary transformation of $\PP^1$-bundles whose center $p$ is a point fixed by $\phi \Autzero(S)\phi^{-1}$, and $T$ is the blow-up of $S'$ at $p$. 
		In the second case, the induced Sarkisov link $S'\times \PP^n \rmap S' \times \PP^n$ is equal to $id_{S'\times \PP^n}$. 
		
		\item The only non-trivial $\Autzero(S\times \PP^n)$-Sarkisov diagrams, where $\pi_1'\colon S'\times \PP^n\to S'$ is the LHS Mori fibre space, are the following ones:
		\[
		\begin{tikzcd} [column sep=3em,row sep = 3em]
			T\times \PP^n\arrow[d,"\eta\times id_{\PP^n}",swap]\arrow[r,equal] & T\times \PP^n\arrow[d,"\pi_1''"] &   S'\times \PP^n\arrow[d,"\pi_1'",swap]\arrow[rr,equal] && S'\times \PP^n\arrow[d,"\pi'"]  \\
			S'\times \PP^n\arrow[d,"\pi_1'",swap] & T\arrow[ld,"\eta"]& S'\arrow[rd,"\tau",swap] && C\times \PP^n\arrow[ld,"p_1"] \\
			S' &&& C &.
		\end{tikzcd}
		\] 
		The induced Sarkisov link $S'\times \PP^n \rmap T \times \PP^n$ is equal to $\eta^{-1} \times id_{\PP^n}$ in the former case and $id_{S'\times \PP^n}$ in the latter, where $\eta\colon T \to S'$ is the blowup of $S'$ at point $p$ fixed by $\phi \Autzero(S)\phi^{-1}$.
	\end{enumerate}
\end{lemma}

\begin{proof}
	\begin{enumerate}[wide]		
		\item\label{item:1} We distinguish between two cases depending on the base $R$ of the diagram: 
		if $R = C\times \PP^n$ then we have a link of Type I or II and so 		
		the first step of the link is an $\Autzero(S\times \PP^n)$-equivariant divisorial contraction $\alpha \colon Y \to S' \times \PP^n$. 
		Note that by \cite[Corollary 4.2.7]{BSU}, it follows that $(\phi\times id_{\PP^n})\Autzero(S\times \PP^n)(\phi\times id_{\PP^n})^{-1}\simeq \phi \Autzero(S)\phi^{-1}\times \PGL_{n+1}(\kk)$.		
		Let $(q,x) \in S' \times \PP^n$ be a point in the center of $\alpha$. 
		If $q$ is not point fixed by $\phi \Autzero(S)\phi^{-1}$, 
		then  and by Lemma \ref{lemmadim2} and the description of $\phi \Autzero(S)\phi^{-1}$, the closure of the orbit of $(q,x)$ is a Cartier divisor and thus $\alpha$ is an isomorphism, contradicting the assumption that $\alpha$ is a divisorial contraction.

		Thus we may assume that $q$ is fixed by $\phi\Autzero(S)\phi^{-1}$. 
		In that case the orbit of $(q,x)$ is precisely $\{q\} \times \PP^n$. Notice that the codimension of $\{q\}\times \PP^n$ is $2$ and so by \cite[Lemma 2.13]{BLZ}
		\[
		\alpha = (\eta \times id_{\PP^n}) \colon T \times \PP^n \to S' \times \PP^n,
		\]
		where $\eta \colon T \to S'$ is the blowup of $S'$ at $q$.
		By Remark \ref{rank 2 fibrations}, the unique Sarkisov diagram containing $T\times \PP^n \to C \times \PP^n$ is the one given in the statement.
		
		We now consider the case when $R \neq C \times \PP^n$. Then we have a contraction $C \times \PP^n \to R$ of relative Picard rank $1$.
		Since $\rho(C \times \PP^n)= 2$, the cone of curves $\NE(C \times \PP^n)$ has two extremal rays and so there are only two such contractions, namely the projections to the two factors: $C \times \PP^n \to C$ and $C \times \PP^n \to \PP^n$.
		However, by property \ref{prop:MMP} of Definition \ref{Sarkisov link}, $C \times \PP^n \to \PP^n$ would have to be an output of some MMP on a klt pair $(Z,\Phi)$, and thus by \cite{HM} its exceptional locus would be rationally connected, a contradiction. Thus $R = C$ and again we conclude by Remark \ref{rank 2 fibrations} for $S'\times \PP^n \to C\times \PP^n$.
		
		\item We again proceed by a similar distinction of cases. If $R =  S'$ then, as in the proof of \ref{item:1}, the first step is an $\Autzero(S \times \PP^n)$-equivariant divisorial contraction $\eta \times id_{\PP^n} \colon T \times \PP^n \to S' \times \PP^n$, where $\eta\colon T\to S'$ is the blow-up of a point of $S'$ fixed by $\phi \Autzero(S) \phi^{-1}$, and we conclude by Remark \ref{rank 2 fibrations}.
		
		If $R \neq S'$, then $S' \to R$ is one of the two morphisms $S' \to C$ or $S' \to \check{S'}$, where the latter is the contraction of the minimal section. 
		Again, by \cite{HM} we may exclude the latter case since its exceptional locus is not rationally connected. 
		Finally, Remark \ref{rank 2 fibrations}, once again, guarantees that the Sarkisov diagram is the one in the statement.
	\end{enumerate}
	
\end{proof}

\begin{lemma}\label{S'xPP^n}
	Let $n\geq 1$ and $C$ be a curve of genus $g\geq 1$. Let $\tau\colon S\to C$ be a decomposable $\PP^1$-bundle such that $\seg(S)< -(1+\deg(K_C))$ with minimal section $\sigma$. Let $\phi\colon S\dashrightarrow S'$ be an $\Autzero(S)$-equivariant birational map, with $S'$ being a smooth projective surface which is not minimal. 
	Denote by $\pi_1'\colon S'\times \PP^n \to S'$ the projection to the first factor. 
	Then the only non-trivial $\Autzero(S\times \PP^n)$-Sarkisov diagrams, where $\pi_1'\colon S'\times \PP^n\to S'$ is the LHS Mori fibre space, are the following ones:
	\[
	\begin{tikzcd} [column sep=3em,row sep = 3em]
		T\times \PP^n\arrow[d,"\eta\times id_{\PP^n}",swap]\arrow[r,equal] & T\times \PP^n\arrow[d,"\pi_1''"]& S'\times \PP^n\arrow[d,"\pi_1'",swap]\arrow[r,equal] & S'\times \PP^n\arrow[d,"\kappa\times id_{\PP^n}"]   \\
		S'\times \PP^n\arrow[d,"\pi_1'",swap] & T\arrow[ld,"\eta"] & S'\arrow[rd,"\kappa",swap] & T\times \PP^n\arrow[d,"\pi_1''"]\\
		S' &&& T. 
	\end{tikzcd}
	\] 
	In the first case, $\eta \colon T \to S'$ is the blow-up of a point $p$ fixed by $\phi \Autzero(S)\phi^{-1}$.
	In the second case, $\kappa \colon S' \to T$ is the contraction of a $(-1)$-curve $l$.
	In both cases, $\pi_1''$ denotes the projection to the first factor.
\end{lemma}

\begin{proof}
	We again distinguish between two cases depending on the base $R$ of the Sarkisov diagram:
	if $R = S'$ then the first step of the link is an $\Autzero(S\times \PP^n)$-equivariant divisorial contraction $\alpha\colon Y \to S' \times \PP^n$.	
	We follow the same strategy of the proof of Lemma \ref{SxPP^n}:
	first by \cite[Corollary 4.2.7]{BSU}, $(\phi\times id_{\PP^n})\Autzero(S\times \PP^n)(\phi\times id_{\PP^n})^{-1}=\phi \Autzero(S)\phi^{-1} \times \PGL_{n+1}(\kk)$.
	This again implies that $\alpha$ has to be an extraction with center of the form $\{q\} \times \PP^n$, where $q$ is a point fixed by the action of $\phi\Autzero(S)\phi^{-1}$ on $S'$.
	Since the center is of codimension $2$, again using \cite[Lemma 2.13]{BLZ}, we conclude that 
	\[
	a = \eta \times id_{\PP^n} \colon T \times \PP^n \to S'\times \PP^n,
	\]
	where $\eta\colon T \to S'$ is the blow-up of $q$.	
	By Remark \ref{rank 2 fibrations}, the diagram is the one given in the statement.
	
	If $R \neq S'$,	we have a morphism $S' \to R$ of relative Picard rank $1$. 
	Since $S'$ is not minimal, its Picard rank is greater or equal to $3$ which already implies that $R=T$ is a surface.
	Again, using Remark \ref{rank 2 fibrations} we may conclude that the diagram is the one proposed in the statement.
	Moreover, by property \ref{prop:singularities} of Definition \ref{Sarkisov link}, $T\times \PP^n$ has to have terminal singularities.
	Thus the singular locus of $T \times \PP^n$ has codimension at least $3$ (see \cite[Corollary 5.18]{KM98}). If $q\in T$ is singular, then $\{q\}\times \PP^n$ is singular and has codimension $2$ in $T\times \PP^n$.
	This implies that $T$ is smooth and consequently, $S' \to T$ is the contraction of a $(-1)$-curve.
\end{proof}

We prove below the higher dimensional analog of Proposition \ref{dim2}.

\begin{theorem}\label{higherdim}
	Let $n\geq 1$. Let $C$ be a curve of genus $g\geq 1$, let $S$ be a decomposable $\PP^1$-bundle over $C$ such that $\seg(S)< -(1+\deg(K_C))$. Then $\Autzero(S\times \PP^n)$ is not contained in a maximal connected algebraic subgroup of $\Bir(S\times \PP^{n})$.
\end{theorem}

\begin{proof}
	Assume that $\Autzero(S\times \PP^n)$ is contained in a maximal connected algebraic subgroup $G\subset \Bir(S\times \PP^n)$. By \cite[Corollary 3]{Brionnormal}, there exists a normal and projective variety $Y$, $G$-birationally equivalent to $S\times \PP^n$, and on which $G$ acts regularly. Then we use an equivariant resolution of singularities (see \cite[Thm. 3.36, Prop. 3.9.1]{Kollar}) to furthermore assume that $Y$ is smooth. Running an MMP, which is $G$-equivariant by \cite[Lemma 2.5]{Floris}, we get an $\Autzero(S\times \PP^n)$-equivariant birational map $\chi\colon S\times \PP^n \dashrightarrow Y$ such that $G\simeq \Autzero(Y)$ and $Y\to B$ is a Mori fibre space.
	By \cite[Theorem 1.2]{Floris}, $\chi$ decomposes as a product of $\Autzero(S\times \PP^n)$-equivariant Sarkisov links. 
	By Lemmas \ref{SxPP^n} and \ref{S'xPP^n}, it follows that $Y= T \times \PP^n$ for some surface $T$ and $\chi$ is of the form $\psi \times id_{\PP^n}$, where $\psi\colon S \rmap T$ is an $\Autzero(S)$-equivariant birational map.
	Up to possibly performing an extra link of Type IV (namely the RHS link in Lemma \ref{SxPP^n} \ref{3.5.1}), we may assume that $B = T$ and $\theta$ is given by the projection to the first factor.	
	Contracting successively all $(-1)$-curves in $T$ yields an $\Autzero(S\times \PP^n)$-equivariant birational map $\phi\times id_{\PP^n}\colon S\times \PP^n \dashrightarrow S'\times \PP^n$ (by Blanchard's lemma, see Proposition \ref{blanchard}), where $\phi$ is $\Autzero(S)$-equivariant and $S'$ is a ruled surface. Two cases arise: either $\phi$ is an isomorphism and $\seg(S)=\seg(S')$, or $\phi$ is not an isomorphism and $\seg(S')<\seg(S)$ by Lemma \ref{lemma2dim2}. In both cases, $\seg(S')\leq \seg(S)$ and since $G$ is maximal, $G$ is isomorphic to $\Autzero(S'\times \PP^n) \simeq  \Autzero(S')\times \PGL_{n+1}(\kk)$ (\cite[Corollary 4.2.7]{BSU}). Let $\phi'\colon S'\dashrightarrow S''$ be an elementary transformation of $S'$ centered at a point on the minimal section. Then $\phi' \Autzero(S')\phi'^{-1} \subsetneq \Autzero(S'')$ by Lemma \ref{lemmadim2}. Thus $(\phi'\times id_{\PP^n}) \Autzero(S'\times \PP^n) (\phi'\times id_{\PP^n})^{-1} \subsetneq \Autzero(S''\times \PP^n)$, which contradicts the maximality of $G$ as connected algebraic subgroup of $\Bir(S\times \PP^{n})$.
\end{proof}

\begin{proof}[\bf{Proof of Theorem A}]
	Let $C$ be a curve of positive genus and $S \to C$ be a ruled surface. As $S$ is birational to $C\times \PP^1$, we get for all $n\geq 1$ that $\Bir(C\times \PP^n)\simeq \Bir(S\times \PP^{n-1})$. We conclude with Proposition \ref{dim2}  for $n=1$ and Theorem \ref{higherdim} for $n\geq 2$.
\end{proof}

\bibliographystyle{alpha} 
\bibliography{bib} 	


\end{document}